\documentclass[11pt]{amsart}
\usepackage{amsmath,amssymb,amsthm}
\usepackage[margin=1.25in]{geometry}
\usepackage{xcolor}
\newtheorem{theorem}{Theorem}[section]
\newtheorem{lemma}[theorem]{Lemma}
\newtheorem{proposition}[theorem]{Proposition}
\newtheorem{corollary}[theorem]{Corollary}
\theoremstyle{definition}
\newtheorem{definition}[theorem]{Definition}
\newtheorem{example}[theorem]{Example}
\theoremstyle{remark}
\newtheorem{remark}[theorem]{Remark}

\newcommand{\cB}{\mathcal B}

\newcommand{\D}{\mathbb D}
\newcommand{\T}{\mathbb T}
\newcommand{\RR}{\mathbb R}

\DeclareMathOperator{\Area}{Area}
\newcommand{\norm}[1]{\lVert #1\rVert}

\title[Hyponormal $\beta(H)$-Class Operators]{Quasinormal and hyponormal operators in the $\beta(H)$-class}
\author[R. Golla]{Ramesh Golla}
\address{Department of Mathematics\\IIT Hyderabad, Kandi\\ Sangareddy, Telangana, 502 284\\ India.}
\email{rameshg@math.iith.ac.in}
\author[H. Osaka]{Hiroyuki Osaka}
\address{ Department of Mathematical Sciences\\ Ritsumeikan University\\ Kusatsu, Shiga, 525-8577, Japan.}

\email[]{osaka@se.ritsumei.ac.jp}

\author[S. S. Sequeira]{Shanola S. Sequeira}
\address{School of Mathematics and Computer Science, IIT Goa, Farmagudi, Goa-403401, India}
\email[]{shanolasequeira@gmail.com}


\date{October 1, 2026}

\begin{document}

\begin{abstract}
Following \cite{BR} we study
hyponormal operators in the strictly larger class $\beta(H)$ of operators
whose restriction to every non-zero reducing subspace is norm
attaining. The Weyl-type characterization of positive operators via the
essential spectrum used in the absolutely norm attaining case is replaced by a
well-ordering condition on the eigenvalues (Theorem~\ref{thm:positive}). We
prove a transfinite, block-diagonal representation $T=\bigoplus_i\lambda_iS_i$
(with isometries $S_i$) for the $\beta(H)$-operators whose $|T|$-eigenspaces are
reducing, equivalently the quasinormal operators with $|T|\in\beta(H)$, and we
show by examples that hyponormality alone does not give such a decomposition.
For normal operators we obtain a complete representation, and for hyponormal
$\beta(H)$-operators we show that normality is equivalent to the spectrum having
planar Lebesgue measure zero. An analogous area-zero criterion for normality is proved for hyponormal operators attaining the minimum modulus on every non-zero reducing subspace.
\end{abstract}

\maketitle

\section{Preliminaries}
The main goal of this note is to describe the class of hyponormal operators that attain their norm on every non-zero reducing subspace. The class of operators that attain their norm on every non-zero reducing subspace is an important subclass of norm-attaining operators. This class, known as the class of $\beta(H)$-class operators, was introduced in \cite{RO21}. Further characterizations and structural results for normal and quasinormal operators in this class were established in \cite{RO22}. An important feature of the $\beta(H)$-class is that every bounded linear operator on a Hilbert space can be approximated by $\beta(H)$-class operators having a non-zero invariant subspace. We refer to \cite{ROS26} for further details and additional results concerning the existence of non-zero hyperinvariant subspaces.

The $\beta(H)$-class contains another important subclass of norm-attaining operators, namely, the class of absolutely norm-attaining operators, or $\mathcal{AN}$-operators. The class of $\mathcal{AN}$-operators has been extensively studied. We refer to \cite{CN,PP,Ram18,VR19} for results concerning their structure, representation, and spectral characterizations, particularly in the cases of positive, self-adjoint, and normal operators. However, the class of $\mathcal{AN}$-operators is not dense in the space of bounded operators with respect to the operator norm. We refer to \cite{RS23} for further details on the operator-norm closure of the class of $\mathcal{AN}$-operators. Our work is motivated in part by the structural results for hyponormal $\mathcal{AN}$-operators obtained in \cite{BR}.

It is well known, as a consequence of Putnam's inequality, that a hyponormal operator whose spectrum has planar area zero is normal. Although the converse does not hold in general, we prove that, within the class of $\beta(H)$-class hyponormal operators, normality is equivalent to the spectrum having planar area zero.

There is a natural counterpart of the $\beta(H)$-class obtained by requiring attainment of the minimum modulus on every non-zero reducing subspace. We refer to \cite{NR26} for results concerning this class. Motivated by this viewpoint, we revisit the characterization of positive operators in the $\beta(H)$-class. In contrast to the usual sequence-based description, the structure of positive $\beta(H)$-class operators naturally involves countable ordinals. We establish this structure first and then use it to obtain a representation theorem for hyponormal $\beta(H)$-class operators.

Throughout, $H$ is an infinite dimensional, complex, separable Hilbert space,
and $\cB(H)$ denotes the bounded linear operators on $H$. For $T\in\cB(H)$
we write the null space and the range spaces of $T$ by $N(T)$ and $R(T)$, respectively. The adjoint of $T$ will be denoted by $T^*$, the spectrum of $T$ by $\sigma(T)$. The modulus of $T$ is defined by  $|T|=(T^*T)^{1/2}$, and
$\sigma_p(T)$ for the set of eigenvalues. For a self-adjoint $S$,
$\sigma_{ess}(S)$ is the set of points of $\sigma(S)$ that are not isolated
eigenvalues of finite multiplicity \cite{Sch}. 

A closed subspace $M$ is
\emph{reducing} for $T$ if $TM\subseteq M$ and $T^*M\subseteq M$, equivalently
if the orthogonal projection onto $M$ commutes with $T$. Orthogonal complements
and closed orthogonal sums of reducing subspaces are reducing, and a reducing
subspace of $T|_M$, for $M$ reducing, is reducing for $T$. A bounded operator $T$ is
called \emph{hyponormal} if $TT^*\le T^*T$ and \emph{quasinormal} if $T$ commutes with
$T^*T$; quasinormal operators are  hyponormal.

If $\{H_i\}$ is a countable family of Hilbert spaces and $T_i\in\cB(H_i)$ with
$\sup_i\norm{T_i}<\infty$, then $\bigoplus_iT_i$ acts coordinatewise on
$\bigoplus_iH_i$ and $\norm{\bigoplus T_i}=\sup\norm{T_i}$ \cite{Con}.

An operator $T\in\cB(H)$ is \emph{norm attaining} if $\norm{Tx}=\norm T\norm x$
for some $x\in H\setminus\{0\}$.

\begin{definition}\label{def:beta}\cite{RO21}
$T\in\cB(H)$ is a \emph{$\beta(H)$-class operator}, written $T\in\beta(H)$, if
$T|_M\colon M\to M$ is norm attaining for every non-zero closed reducing
subspace $M$ of $T$.
\end{definition}

Since $H$ itself is reducing, every $T\in\beta(H)$ is norm attaining.
Recall that $T$ is \emph{absolutely norm attaining}, $T\in\mathcal{AN}(H)$, if
$T|_M\colon M\to H$ is norm attaining for \emph{every} non-zero closed
subspace $M$ \cite{CN}. For a reducing subspace $M$ we have $TM\subseteq M$, so
the two tests ``$T|_M\colon M\to H$ attains its norm'' and ``$T|_M\colon M\to M$
attains its norm'' coincide. Hence
\[
\mathcal{AN}(H)\subseteq\beta(H);
\]
the inclusion is strict (See~ \cite{RO21} for more details).

\begin{remark}
    If $T\in \mathcal B(H)$ is irreducible, that is , $T$ has no reducing subspaces other than ${\{0}\}$ and $H$, then $T\in \beta(H)$ if and only if $T$ is norm attaining. 
\end{remark}

We refer to \cite{RO21,RO22} for the structure of positive operators in
$\beta(H)$. Similar to the denseness of norm attaining operators, the $\beta(H)$-class of operators is dense in $\mathcal B(H)$ with respect to the operator norm. In fact, more is true;  a bounded operator can be approximated in operator norm by a
$\beta(H)$-class operator having a  non-trivial invariant subspace. We refer to \cite{ROS26} for more details of this result and concerning hyperinvariant subspaces of $\beta(H)$-class operators.

In the present article we give a representation theorem for hyponormal $\beta(H)$-class operators in the spirit of \cite{BR}. Since hyponormality alone
does not force a block-diagonal structure (Examples~\ref{ex:shift} and
\ref{ex:sum}), our decomposition (Theorem~\ref{thm:main}) is stated for those
$\beta(H)$-operators whose $|T|$-eigenspaces reduce $T$; in particular it covers
all normal $\beta(H)$-operators. As a consequence we obtain a
description of normal operators in $\beta(H)$ and a normality criterion for
hyponormal $\beta(H)$-operators. Finally, we prove the corresponding results for the class minimum attaining operators on every non zero reducing subspace.  We refer to \cite{NR26} for results concerning this class. Motivated by this viewpoint, we revisit the characterization of positive operators in the $\beta(H)$-class. In contrast to the usual sequence-based description, the structure of positive $\beta(H)$-class operators naturally involves countable ordinals. We establish this structure first and then use it to obtain a representation theorem for hyponormal $\beta(H)$-class operators. As applications, we characterize normal operators in $\beta(H)$ and establish a normality criterion for hyponormal $\beta(H)$-class operators based on Putnam's inequality. We also obtain corresponding results for the class $\mathcal M_r(H)$.

\section{Structure of hyponormal $\beta(H)$-class operators}

For $S\ge0$ every closed $S$-invariant subspace is reducing (since $S$ is
self-adjoint), so membership in $\beta(H)$ tests exactly the $S$-invariant
subspaces. Let $E$ be the spectral measure of $S$ and
$\Lambda(S)=\{\lambda:E(\{\lambda\})\ne0\}=\sigma_p(S)$.

The following fact will be used throughout.
\begin{proposition}\cite[Proposition 2.4]{CN}
If $T\in \mathcal B(H)$ is positive, then $T$ is norm attaining if and only if
$\norm T\in\sigma_p(T)$.     
\end{proposition}

\begin{theorem}\label{thm:positive}
Let $S\in\cB(H)$, $S\ge0$. Then $S\in\beta(H)$ if and only if
\begin{enumerate}
\item[(a)] $S$ has pure point spectrum, i.e.\ $H=\bigoplus_{\lambda\in\Lambda(S)}E(\{\lambda\})H$; and
\item[(b)] $\Lambda(S)$ is well-ordered by $\ge$, equivalently, it contains no infinite strictly increasing sequence.
\end{enumerate}
\end{theorem}

\begin{proof}
($\Rightarrow$) Suppose $\lambda_1<\lambda_2<\cdots$ lie in $\Lambda(S)$ and put
$M=\bigoplus_nE(\{\lambda_n\})H$, a reducing subspace with
$\sigma_p(S|_M)=\{\lambda_n\}$. Then $\norm{S|_M}=\sup_n\lambda_n$, which is not
an eigenvalue of $S|_M$; so $S|_M$ is not norm attaining, a contradiction. This
proves (b).

Suppose now that (b) holds but (a) fails. Let $H_p=\bigoplus_\lambda
E(\{\lambda\})H$ and $H_c=H_p^\perp\ne {\{0}\}$. Both are reducing and $S|_{H_c}$
has no eigenvalues, so by the elementary fact it is not norm attaining,
contradicting $S\in\beta(H)$.

($\Leftarrow$) Let $M\ne {\{0}\}$ be a closed reducing 
subspace. It is invariant under every spectral projection $E(B)$ (for $B$ in the Borel $\sigma$-algebra generated by $\sigma(S)$) and, by (a),
$M=\bigoplus_{\lambda\in I}M_\lambda$ with $M_\lambda=M\cap E(\{\lambda\})H\ne {\{0}\}$
for $\lambda\in I\subseteq\Lambda(S)$, $I\ne\emptyset$. By (b), $I$ has a largest
element $\lambda^*$, and every unit vector $x\in M_{\lambda^*}$ satisfies
$\norm{S|_Mx}=\lambda^*=\norm{S|_M}$.
\end{proof}

\begin{remark}\label{rem:countable}
Every well-ordered subset of $\RR$ is countable (an uncountable one would give
uncountably many pairwise disjoint intervals with rational endpoints). Hence
$\Lambda(S)$ is countable, and is order-isomorphic to a countable ordinal
$\gamma$: $\Lambda(S)=\{\lambda_i\}_{i<\gamma}$, strictly decreasing in $i$.
\end{remark}

\begin{remark}\label{rem:strict}
\begin{enumerate}
\item By the Pandey--Paulsen theorem \cite[Theorem 5.11]{PP}, a positive $P\in\mathcal{AN}(H)$ has
the form $P=\alpha I+K+F$ with $K\ge0$ compact and $F$ self-adjoint of finite
rank; the eigenvalues approach the single essential value $\alpha$ from above
apart from finitely many exceptions, so $P$ satisfies the conditions of
Theorem~\ref{thm:positive}, consistent with $\mathcal{AN}(H)\subseteq\beta(H)$.
Theorem~\ref{thm:positive} permits several, even countably many, accumulation points of eigenvalues (and eigenspaces of infinite dimension) instead of the
single value $\alpha=\min\sigma_{ess}(P)$ forced by \cite[Theorem~1.6]{BR},
provided every accumulation is approached from above. For example,
the diagonal operator with simple eigenvalues $1+\frac1k$ ($k\ge1$) and
$\frac12+\frac1k$ ($k\ge3$) is in $\beta(H)$ by Theorem~\ref{thm:positive},
but not in $\mathcal{AN}(H)$ because $\sigma_{ess}=\{1,\frac12\}$ is not a
singleton. So $\mathcal{AN}(H)\subsetneq\beta(H)$.
\item In \cite{RO21}, necessary and sufficient conditions for a positive operator to belong to $\beta(H)$ were established, together with a representation theorem for such operators. In the course of studying operators that attain their minimum modulus on every nonzero reducing subspace, as investigated in \cite{NR26}, the authors observed that a characterization in terms of sequences alone does not adequately capture the structure of positive operators in $\beta(H)$. Theorem~\ref{thm:positive} provides the appropriate formulation.

It is also worth noting that, for an arbitrary operator $T\in\mathcal B(H)$, the modulus $|T|$ need not belong to $\beta(H)$ whenever $T$ does. Indeed, the reducing subspaces of $|T|$ may be larger in number than those of $T$, so norm attainment on every nonzero reducing subspace of $T$ does not automatically imply the corresponding property for $|T|$. See example \ref{ex:sum}.
\end{enumerate}
\end{remark}

The following result used frequently throughout.
\begin{lemma}\label{lem:attain}\cite{CN}
If $T\in\cB(H)$ is norm attaining and $H\ne\{0\}$, then $\norm T\in\sigma_p(|T|)$.
\end{lemma}

\begin{lemma}\label{lem:top}
Let $T\in\cB(H)$ be hyponormal and norm attaining, $T\ne0$, let
$\lambda=\norm T$, and $M_0=N(|T|-\lambda I)$. Then $M_0\neq {\{0}\}$ and
\begin{enumerate}
\item[(i)] $TM_0\subseteq M_0$ and $T|_{M_0}=\lambda S_0$ with $S_0\in\cB(M_0)$ an isometry.
\item[(ii)] With respect to $H=M_0\oplus M_0^\perp$, we have $$T=\begin{pmatrix}\lambda S_0&A\\0&B\end{pmatrix}$$ with $R(A)\subseteq N(S_0^*)$.
\item[(iii)] $M_0$ reduces $T$ if and only if $A=0$. This holds if $S_0$ is
unitary (e.g.\ if $\dim M_0<\infty$ or $T$ is normal). It also holds if
$T$ is quasinormal, since then $T$ commutes with $|T|$ and hence
$M_0=N(|T|-\lambda I)$ reduces $T$; in that case $S_0$ need not be unitary.
\item[(iv)] Write $S_0=U_0\oplus V_0$ on $M_0=W_0\oplus L_0$ with $U_0$ unitary and
$V_0$ a unilateral shift (Wold decomposition). Then $W_0$ reduces $T$.
\end{enumerate}
\end{lemma}

\begin{proof}
(i) $M_0\ne {\{0}\}$ by Lemma~\ref{lem:attain}. For $x\in M_0$ we have $T^*Tx=\lambda^2x$,
so $\norm{Tx}=\lambda\norm x$. Put $y=Tx$. Then $T^*y=\lambda^2x$, hence
$\norm{T^*y}=\lambda^2\norm x=\lambda\norm y$. As $T$ is hyponormal. we have
$\norm{Ty}\ge\norm{T^*y}=\lambda\norm y\ge\norm{Ty}$, so $\norm{Ty}=\lambda\norm y$,
i.e.\ $\langle(\lambda^2-T^*T)y,y\rangle=0$, whence $y\in M_0$ as before. Thus
$TM_0\subseteq M_0$, and $\lambda^{-1}T|_{M_0}$ is an isometry (cf.\ \cite{Lee}).

(ii) Invariance of $M_0$ gives the zero in the lower left corner. Compute the
$(1,1)$ entries: $T^*T$ contributes $\lambda^2I_{M_0}$ and $TT^*$ contributes
$\lambda^2S_0S_0^*+AA^*$ for some $A\in \mathcal B(M_0^{\perp},M_0)$. Hence hyponormality yields
\[
0\le\lambda^2(I-S_0S_0^*)-AA^*.
\]
For $x\in R(S_0)$ the right side gives $\langle AA^*x,x\rangle\le0$, so $A^*x=0$.
Thus $R(A)\perp R(S_0)$, i.e.\ $R(A)\subseteq N(S_0^*)$.

(iii) $T^*=\begin{pmatrix}\lambda S_0^*&0\\A^*&B^*\end{pmatrix}$, so $M_0$ is
$T^*$-invariant iff $A^*=0$. If $S_0$ is unitary, $N( S_0^*)=0$ and $A=0$.

(iv) For $x\in W_0$, $Tx=\lambda U_0x\in W_0$ and
$T^*x=\lambda S_0^*x+A^*x$, where $S_0^*x=U_0^*x\in W_0$ and $A^*x=0$ because
$R(A)\subseteq N(S_0^*)\subseteq L_0\perp W_0$.
\end{proof}
\begin{theorem}\label{thm:main}
For $T\in\cB(H)$, $T\ne0$, the following are equivalent.
\begin{enumerate}
\item[(a)] $T\in\beta(H)$ and $N(|T|-\lambda I)$ reduces $T$ for every $\lambda>0$.
\item[(b)] $T$ is quasinormal and $|T|\in\beta(H)$.
\item[(c)] There are a countable ordinal $\gamma$, a strictly decreasing family
$(\lambda_i)_{i<\gamma}$ of positive reals, and a decomposition
\[
H=\Big(\bigoplus_{i<\gamma}M_i\Big)\oplus N(T)
\]
into pairwise orthogonal reducing subspaces of $T$ such that
\[
T=\Big(\bigoplus_{i<\gamma}\lambda_iS_i\Big)\oplus0_{N(T)},\qquad
M_i=N(|T|-\lambda_iI),
\]
where each $S_i\in\cB(M_i)$ is an isometry.
\end{enumerate}
In particular, $T$ has no off-diagonal (finite-rank correction) term, and
$|T|=\big(\bigoplus_i\lambda_iI_{M_i}\big)\oplus0$. Here $N(T)=\{0\}$ is allowed.
\end{theorem}

\begin{proof}
(a)$\Rightarrow$(c). Define recursively, for ordinals $i$, the subspaces
\[
H_i=\Big(\bigoplus_{j<i}M_j\Big)^\perp,
\]
and, as long as $T|_{H_i}\ne0$,
\[
\lambda_i=\norm{T|_{H_i}},\qquad M_i=N(|T|-\lambda_iI).
\]
We check by transfinite induction that this is well defined and that
\begin{enumerate}
\item[(1)] $H_i$ reduces $T$;
\item[(2)] $\lambda_i>0$ is an eigenvalue of $|T|$ with eigenvector in $H_i$, and $\lambda_i<\lambda_j\; (j<i)$;
\item[(3)] $M_i\subseteq H_i$, and $M_i$ reduces $T$.
\end{enumerate}
(1) holds because the $M_j$, $j<i$, reduce $T$. By (1) and $T\in\beta(H)$,
$T|_{H_i}$ is norm attaining, so Lemma~\ref{lem:attain} applied to $T|_{H_i}$,
whose modulus is $|T||_{H_i}$, shows that $\lambda_i$ is an eigenvalue of $|T|$
with an eigenvector in $H_i$. For $j<i$ we have $H_i\subseteq H_{j+1}\subseteq H_j$
and $H_{j+1}=H_j\ominus M_j$ (by (3) for $j$); so $\lambda_i\le\lambda_{j+1}\le\lambda_j$.
Moreover $\lambda_{j+1}\ne\lambda_j$, since $\lambda_{j+1}$ has an eigenvector in
$H_{j+1}\perp M_j=N(|T|-\lambda_j)$. Hence $\lambda_i<\lambda_j$, which proves (2).
Since $\lambda_i\ne\lambda_j$, $N(|T|-\lambda_i)\perp N(|T|-\lambda_j)=M_j$ for all
$j<i$, so $M_i\subseteq H_i$; and $M_i$ reduces $T$ by hypothesis (a). This proves (3).

On $M_i$ we have $T^*T=\lambda_i^2I$, and $TM_i\subseteq M_i$, so
$S_i:=\lambda_i^{-1}T|_{M_i}$ is an isometry in $\cB(M_i)$.
Since the $\lambda_i$'s are distinct elements of $\mathbb R$, Remark \ref{rem:countable} shows that the index set $\gamma$ is countable. Hence the recursion terminates at a countable ordinal with $T|_{H_\gamma}=0$. Then
$H_\gamma \subseteq N(T)$. Conversely, $N(T)=N(|T|)\perp M_j$ for all $j<\gamma$ gives $N(T)\subseteq H_\gamma$. Thus $H_\gamma=N(T)$ and
$H=(\bigoplus_{i<\gamma}M_i)\oplus N(T)$.

(c)$\Rightarrow$(b). $T^*T=(\bigoplus\lambda_i^2I_{M_i})\oplus0$ commutes with
$T$, which respects the decomposition; so $T$ is quasinormal. Also
$|T|=(\bigoplus\lambda_iI_{M_i})\oplus0$ has pure point spectrum with
$\Lambda(|T|)=\{\lambda_i\}_{i<\gamma}\cup(\{0\}\text{ if }N(T)\ne0)$, which is
well-ordered by $\ge$. By Theorem~\ref{thm:positive}, $|T|\in\beta(H)$.

(b)$\Rightarrow$(a). 

Suppose that $M\neq\{0\}$ reduces $T$. Then $M$ also reduces $|T|$.
Since $|T|\in\beta(H)$, the restriction $|T||_M$ attains its norm.
Thus, there exists $x\in M$ with $\|x\|=1$ such that
$\bigl\||T|x\bigr\|=\bigl\||T||_M\bigr\|$. Since $|T||_M$ is positive, norm attainment implies that
$|T|x=\bigl\||T||_M\bigr\|x$. Moreover,
$\|Tx\|=\bigl\||T|x\bigr\|$.
Consequently,
\[
\|Tx\|=\bigl\||T||_M\bigr\|=\|T|_M\|.
\]
Therefore, $T|_M$ attains its norm. Hence $T\in\beta(H)$.

$T$ commutes with $|T|$, and so does $T^*$. Hence every spectral subspace of $|T|$, in particular
$N(|T|-\lambda I)$, reduces $T$. 
\end{proof}

\begin{remark}
In (a), hyponormality is a consequence rather than an assumption. For quasinormal
$T$ one has $T^*T=|T|^2$ commuting with $T$, so (b) is the cleanest formulation.
\end{remark}

The hypothesis that the $|T|$-eigenspaces reduce $T$ cannot be dropped: the
following examples show that hyponormal operators in $\beta(H)$ need not
decompose as in Theorem~\ref{thm:main}.

\begin{example}\label{ex:shift}
Let $W$ be the weighted shift on $\ell^2$ with orthonormal basis
$e_{-1},e_0,e_1,\dots$ given by $$We_{-1}=\frac12e_0\; \text{and}\; We_n=e_{n+1}\; \text{for}\;n\ge0.$$
Then $W^*W=\operatorname{diag}(\frac14,1,1,\dots)$ and
$WW^*=\operatorname{diag}(0,\frac14,1,1,\dots)$, so
$$W^*W-WW^*=\operatorname{diag}(\frac14,\frac34,0,0,\dots)\ge0.$$ That is, $W$ is hyponormal. Let $P$ be the projection onto a reducing subspace $N$. It
commutes with $W^*W-WW^*$, whose eigenvalue $\frac14$ has eigenspace
$\mathbb Ce_{-1}$; so $Pe_{-1}=ce_{-1}$ with $c\in\{0,1\}$. If $c=1$, then
$W^ne_{-1}\in N$ for all $n$, so $N=\ell^2$. If $c=0$, the same argument for
$N^\perp$ gives $N=0$. Thus $W$ has no non-trivial reducing subspace, and since
$\norm{We_0}=\norm W=1$, $W\in\beta(H)$. But
$N(|W|-I)=\overline{\operatorname{span}}\{e_n:n\ge0\}$ does not reduce $W$
(because $W^*e_0=\frac12e_{-1}$), and $W$ is not quasinormal. Here $\sigma(W)=\overline{\D}$ ($W$ is a finite rank perturbation of the unilateral shift), so $\Area\sigma(W)>0$.
\end{example}

\begin{example}\label{ex:sum}
For $k\ge1$ let $T_k$ be as in Example~\ref{ex:shift} with weight
$\alpha_k=(1-3^{-k})^{1/2}$ in place of $\frac12$, and $T=\bigoplus_kT_k$. The
numbers $\alpha_k^2$ and $1-\alpha_k^2=3^{-k}$ ($k\ge1$) are non-zero and pairwise
distinct, so the eigenvalue $\alpha_k^2$ of $T^*T-TT^*$ has the one dimensional
eigenspace spanned by $e^{(k)}_{-1}$. As in Example~\ref{ex:shift}, every
reducing subspace of $T$ has the form $\bigoplus_{k\in K}H_k$, and
$T|_N$ has norm $1$, attained at $e_0^{(k)}$. So $T$ is hyponormal and
$T\in\beta(H)$. However, $\Lambda(|T|)\supseteq\{\alpha_k\}$ is strictly
increasing, so $|T|\notin\beta(H)$ by Theorem~\ref{thm:positive}(b); explicitly,
$M=\overline{\operatorname{span}}\{e^{(k)}_{-1}\}$ reduces $|T|$ and
$\sup_k\alpha_k=1$ is not attained on $M$. Thus $T\in\beta(H)$ hyponormal does
not imply $|T|\in\beta(H)$, and $T$ is not as in Theorem~\ref{thm:main}.
\end{example}

\begin{remark}
By Lemma~\ref{lem:top} and the standard decomposition of a hyponormal
operator into a normal part and a completely non-normal part,
$T=T_n\oplus T_p$ with $T_n,T_p\in\beta(H)$ (the class $\beta$ passes to
reducing subspaces). The normal part is described by Corollary~\ref{cor:normal} below. A description of the completely non-normal part $T_p$ appears to require further investigation; Example~\ref{ex:shift} illustrates the difficulty.
\end{remark}

Next we characterize normal  $\beta(H)$-class.
\begin{corollary}\label{cor:normal}
Let $T\in\cB(H)$, $T\ne0$, be normal. Then $T\in\beta(H)$ if and only if
\[
T\cong\Big(\bigoplus_{i<\gamma}\lambda_iU_i\Big)\oplus0,
\]
with $\gamma$ a countable ordinal, $(\lambda_i)_{i<\gamma}$ strictly decreasing
positive reals, and $U_i$ unitary on $M_i=N(|T|-\lambda_iI)$, $H=(\bigoplus M_i)\oplus N(T)$.
\end{corollary}

\begin{proof}
Let $T\in\beta(H)$ be normal. Then $T$ and $T^*$ commute with $|T|$, so all
eigenspaces of $|T|$ reduce $T$; hence (a) of Theorem~\ref{thm:main} holds and
(c) gives $T=(\bigoplus\lambda_iS_i)\oplus0$. Each $T|_{M_i}=\lambda_iS_i$ is
normal (restriction to a reducing subspace) and $S_i$ is an isometry, so
$S_i^*S_i=S_iS_i^*=I$, i.e.\ $S_i$ is unitary. The converse follows from
Theorem~\ref{thm:main} ((c)$\Rightarrow$(a)), since such an operator is normal.
\end{proof}
 \begin{remark}
Theorem \ref{thm:main} provides the correct formulation of \cite[Theorem 3.9]{RO22}, while Corollary \ref{cor:normal} recovers \cite[Theorems 3.1 and 3.11]{RO22} in corrected forms. In particular, \cite[Theorem 3.11]{RO22} is formulated in terms of sequences, but this formulation is insufficient to establish the conclusion of Theorem \ref{thm:positive}. Corollary \ref{cor:normal} provides the appropriate formulation, thereby rectifying this issue.
\end{remark}

\begin{lemma}\label{lem:closure}
Let $\Lambda\subseteq\mathbb{R}$ be well ordered by $\geq$. Then
$\overline{\Lambda}$ is also well ordered by $\geq$. In particular,
$\overline{\Lambda}$ is countable, and every point of
$\overline{\Lambda}\setminus\Lambda$ is a limit of points of $\Lambda$
from above only. If $\Lambda$ is well ordered by $\leq$, the same holds
with $\overline{\Lambda}$ well ordered by $\leq$ and with limits from
below only.
\end{lemma}

\begin{proof}
A subset of $\mathbb{R}$ is well ordered by $\geq$ if and only if it
contains no infinite strictly increasing sequence. Suppose, for a
contradiction, that $x_1<x_2<\cdots$ is a sequence in
$\overline{\Lambda}$. Choose $\varepsilon_n>0$ such that
\[
\varepsilon_n<\tfrac13\,(x_{n+1}-x_n)
\quad\text{and}\quad
\varepsilon_n<\tfrac13\,(x_n-x_{n-1})\ \ (n\ge 2),
\]
and pick $\lambda_n\in\Lambda$ with $|\lambda_n-x_n|<\varepsilon_n$.
Then
\[
\lambda_n< x_n+\varepsilon_n< x_{n+1}-\varepsilon_{n+1}<\lambda_{n+1},
\]
so $(\lambda_n)$ is a strictly increasing sequence in $\Lambda$,
contradicting the well ordering of $\Lambda$. Hence $\overline{\Lambda}$
is well ordered by $\geq$, and by Remark~\ref{rem:countable} it is
countable.

If $p\in\overline{\Lambda}\setminus\Lambda$ were the limit of a
sequence in $\Lambda$ increasing to $p$, that sequence would be an
infinite strictly increasing sequence in $\Lambda$. So $p$ is
approached from above only.

The case where $\Lambda$ is well ordered by $\leq$ follows by applying
the above to $-\Lambda$.
\end{proof}

\begin{theorem}\label{thm:criterion}
Let $T\in\beta(H)$ be hyponormal. Then the following are equivalent.
\begin{enumerate}
\item[(1)] $T$ is normal.
\item[(2)] $\Area(\sigma(T))=0$, where $\Area$ is planar Lebesgue measure.
\end{enumerate}
If moreover $T$ is as in Theorem~\ref{thm:main}, $T=(\bigoplus\lambda_iS_i)\oplus0$,
then (1) and (2) are also equivalent to
\begin{enumerate}
\item[(3)] $S_i$ is unitary for every $i<\gamma$.
\end{enumerate}
\end{theorem}

\begin{proof}
(2)$\Rightarrow$(1): Putnam's inequality \cite{Put} gives
$\norm{T^*T-TT^*}\le\frac1\pi\Area(\sigma(T))=0$.

(1)$\Rightarrow$(2): By Corollary~\ref{cor:normal}, $T\cong(\bigoplus\lambda_iU_i)\oplus0$.
Let $K=\{\lambda_i: i<\gamma\}\cup\{0\}$. By Lemma \ref{lem:closure}, $\overline{K}$ is countable. If $|z|\notin \overline{K}$, put $d=\operatorname{dist}(|z|,\overline{K})>0$.
Since $\sigma(\lambda_iU_i)\subseteq\lambda_i\T$ and $\lambda_iU_i$ is normal,
$\norm{(\lambda_iU_i-z)^{-1}}\le\big||z|-\lambda_i\big|^{-1}\le d^{-1}$, and on
$N(T)$ the resolvent norm is $|z|^{-1}\le d^{-1}$. So the direct sum of the
resolvents is bounded. Hence $z\notin\sigma(T)$, and
$\sigma(T)\subseteq\{z:|z|\in \overline{K}\}$. Since $\sigma(T)$ is contained in countable union of circles, it has
planar measure zero.

For the last assertion, (3)$\Rightarrow$(1) is clear, and if $T$ is normal then
each $\lambda_iS_i$ is normal, so the isometry $S_i$ is unitary.
\end{proof}

\begin{remark}
(i) If $T$ is compact, then $T\in\mathcal{AN}(H)\subseteq\beta(H)$ and
$\sigma(T)$ is countable, so a hyponormal compact operator is normal by
Theorem~\ref{thm:criterion}; this recovers \cite[Corollary~3.10]{BR} without
using any decomposition.

(ii) In contrast to \cite[Corollary~3.11]{BR}, no hypothesis such as
$\sigma_{ess}(T)=\omega(T)$ is required here: for hyponormal $\beta(H)$-operators,
normality is equivalent to $\Area\sigma(T)=0$. Note that pure point spectrum of
$|T|$ alone does \emph{not} force $\Area\sigma(T)=0$: the operator $W$ of
Example~\ref{ex:shift} has $\sigma(W)=\overline\D$.
\end{remark}

\subsection{The class $\mathcal M_r(H)$}

Let $T\in \mathcal B(H)$. Then the minimum modulus of $T$ is defined by  $$m(T)=\inf_{\norm x=1}\norm{Tx}.$$ 
The operator $T$ is said to be minimum attaining if there exists $x\in H$ with $\|x\|=1$ such that $\|Tx\|=m(T)$. Let $\mathcal M(H)$ denote the class of minimum attaining operators on $H$. We refer to \cite{CN14} for more details on this class.

The following new class of $\mathcal M(H)$ is studied in \cite{NR26}:
\[
\mathcal M_r(H)=\{T\in\cB(H):T|_M\in\mathcal M(M)\ \text{for every non-zero closed reducing subspace }M\text{ of }T\}.
\]

\begin{corollary}\label{cor:Mr-positive}
For $T\ge0$: $T\in\mathcal M_r(H)$ if and only if $T$ has pure point spectrum and
$\Lambda(T)$ is well-ordered by $\le$, i.e.\ contains no infinite strictly
decreasing sequence. In that case $T=\bigoplus_{i<\gamma}\mu_iI_{M_i}$ with
$(\mu_i)$ strictly increasing and $M_i=N(T-\mu_iI)$.
\end{corollary}

There is an interesting relation between positive $\beta(H)$-class operators and $\mathcal M_r(H)$-class operators. Here we state the result.
\begin{proposition}\label{prop:Mr-beta} \cite[Corollary~4.10]{NR26}
Let $T\ge0$ and $\norm T\le\delta$. Then $T\in\mathcal M_r(H)$ if and only if
$\delta I-T\in\beta(H)$.
\end{proposition}
The following result characterizes normal operators in $\mathcal M_r(H)$.
\begin{theorem}\label{thm:Mr-normal} \cite[Theorem~6.7]{NR26}
Let $T\in\cB(H)$ be normal. Then $T\in\mathcal M_r(H)$ if and only if
$T\cong\bigoplus_{i<\gamma}\mu_iU_i$ with $\gamma$ a countable ordinal,
$(\mu_i)_{i<\gamma}$ strictly increasing in $[0,\infty)$ and $U_i$ unitary on
$M_i=N(|T|-\mu_iI)$, $H=\bigoplus_iM_i$.
\end{theorem}

\begin{theorem}\label{thm:Mr-criterion}
Let $T\in\mathcal M_r(H)$ be hyponormal. Then $T$ is normal if and only if $\Area(\sigma(T))=0$.
\end{theorem}

\begin{proof}
($\Leftarrow$) Putnam's inequality as in Theorem~\ref{thm:criterion}.

($\Rightarrow$) By Theorem~\ref{thm:Mr-normal}, $T\cong\bigoplus\mu_iU_i$; let
$K=\{\mu_i\}$, countable by Lemma~\ref{lem:closure} (applied to the
set $\{\mu_i\}$, well-ordered by $\le$). Exactly as in the proof of
Theorem~\ref{thm:criterion}, $\sigma(T)\subseteq\{z:|z|\in K\}$, which has area zero.
\end{proof}

\begin{example}
Hyponormality alone does not give a block-diagonal decomposition in $\mathcal M_r(H)$
either. Let $T$ be the weighted shift $Te_n=\alpha_ne_{n+1}$ ($n\ge0$) with
$\alpha_n=2-\frac1{n+1}$. The weights increase, so $T$ is hyponormal.
$P$ reducing commutes with $T^*T=\operatorname{diag}(\alpha_n^2)$, whose entries
are pairwise distinct, so $P$ is diagonal; commuting with $T$ forces the
range of $P$ to be $0$ or $H$. Hence $T$ has no non-trivial reducing subspace, and
$m(T)=\alpha_0=1$ is attained at $e_0$, so $T\in\mathcal M_r(H)$. But the
eigenspaces of $|T|$ are the lines $\mathbb Ce_n$, none of which reduces $T$.
\end{example}

\begin{center}
Acknowledgements
\end{center}
The authors used Claude Sonnet 5.5 (Anthropic) for proof-checking suggestions and language and LaTeX editing. The authors reviewed and verified all content and take full responsibility for the manuscript.

\end{document}